\documentclass[oneside,12pt]{amsart}

\usepackage{amscd,amsmath,latexsym}
\usepackage[mathcal]{euscript}

\newtheorem{thm}{Theorem}[section]
\newtheorem{cor}[thm]{Corollary}
\newtheorem{lem}[thm]{Lemma}
\newtheorem{prop}[thm]{Proposition}

\theoremstyle{definition}
\newtheorem{defin}[thm]{Definition}

\theoremstyle{definition}
\newtheorem{exm}[thm]{Example}

\theoremstyle{remark}

\newcommand{\D}{{\delta}}

\newcommand{\R}{{\mathbb R}}
\newcommand{\C}{{\mathbb C}}
\newcommand{\Z}{{\mathbb Z}}

\begin{document}

\title{On decomposing suspensions of simplicial spaces}

\author[A.~Adem]{A.~Adem}
\address{Department of Mathematics
University of British Columbia, Vancouver, B.~C., Canada}
\email{adem@math.ubc.ca}

\author[A.~Bahri]{A.~Bahri}
\address{Department of Mathematics,
Rider University, Lawrenceville, NJ 08648, U.S.A.}
\email{bahri@rider.edu}

\author[M.~Bendersky]{M.~Bendersky}
\address{Department of Mathematics
CUNY,  East 695 Park Avenue New York, NY 10065, U.S.A.}
\email{mbenders@xena.hunter.cuny.edu}

\author[F.~R.~Cohen]{F.~R.~Cohen}
\address{Department of Mathematics,
University of Rochester, Rochester, NY 14625, U.S.A.}
\email{cohf@math.rochester.edu}

\author[S.~Gitler]{S.~Gitler}
\address{Department of Mathematics,
Cinvestav, San Pedro Zacatenco, Mexico, D.F. CP 07360 Apartado
Postal 14-740, Mexico} \email{sgitler@math.cinvestav.mx}

\thanks{A.Adem was partially supported by the NSF and NSERC. A.Bahri was
partially supported by the award of a research leave from Rider
University. F.Cohen was partially supported by NSF grant number
0340575 and DARPA grant number 2006-06918-01. S.Gitler was partially
supported by the Department of Mathematics at Princeton University.}

\date{\today}

\begin{abstract}
Let $X_{\bullet}$ denote a simplicial space. The purpose of this
note is to record a decomposition of the suspension of the
individual spaces $X_n$ occurring in $X_{\bullet}$ in case the
spaces $X_n$ satisfy certain mild topological hypotheses and where
these decompositions are natural for morphisms of simplicial spaces.
In addition, the summands of $X_n$ which occur after one suspension
are stably equivalent to choices of filtration quotients of the
geometric realization $|X_{\bullet}|$. The purpose of recording
these decompositions is that they imply decompositions of the single
suspension of certain spaces of representations \cite{ac,acg} as
well as other varieties and are similar to decompositions of
suspensions of moment-angle complexes
\cite{bahri.bendersky.cohen.gitler} which appear in a different
context.

\end{abstract}

\maketitle

\section{Introduction and Statement of Results}\label{sec:Introduction}

Let $X_{\bullet}$ denote a simplicial space. The purpose of this
note is to give a decomposition of the suspension of the individual
spaces $X_n$ occurring in $X_{\bullet}$ in case the spaces $X_n$
satisfy certain mild topological hypotheses. These decompositions
are natural for morphisms of simplicial spaces. In addition, the
summands of $X_n$ which occur after one suspension are stably
equivalent to choices of filtration quotients of the geometric
realization $|X_{\bullet}|$.

These structures occur in several contexts in useful ways and the
following spaces admit decompositions of the type
discussed above.
\begin{enumerate}
\item The suspension of the the loop space for a (path-connected) suspension of a CW-complex $Y$
  is homotopy equivalent to a bouquet of the suspension of smash products of $Y$ \cite{james,milnor}.

\item Spaces of ordered commuting $n$-tuples in a Lie
group $G$, $Hom(\oplus_n~ \mathbb Z,G)$, assemble to
give a simplicial space
denoted $Hom( \mathbb Z_{\bullet},G)$. 
If $G$ is a closed subgroup of $GL_r(\C)$,
there are natural homotopy equivalences

$$\Sigma  Hom(\oplus_n~ \mathbb Z,G)\to
\bigvee_{1 \leq k \leq n}\Sigma  \bigvee^{\binom n k}
Hom(\oplus_k~ \mathbb Z,G)/ S_k(G)$$
where $S_k(G)$ 
denotes the \textsl{singular subspace} defined as those
commuting $k$--tuples where at least one entry is equal
to $1$ (see \cite{ac}).
The associated spaces of
representations $$Rep(\oplus_n~ \mathbb Z,G) = Hom(\oplus_n~\mathbb
Z,G)/G^{ad}$$ where $G$ acts by conjugation also assemble into a
simplicial space, with similar decompositions (see \cite{acgo}). 
For $G$ a finite group, the simplicial spaces $Hom( \mathbb
Z_{\bullet},G)$ and $Rep( \mathbb Z_{\bullet},G)$ have natural
connections with the cohomology of finite groups
\cite{acg}.

\item The suspension of moment-angle complexes as well as their generalizations are
homotopy equivalent to a bouquet of \textsl{smash product moment-angle
complexes} (see \cite{bahri.bendersky.cohen.gitler}).

\item A compact real algebraic variety as given in \S 5
is homeomorphic to the geometric realization of a simplicial space
$V_{\bullet}$ for which each space $V_n$ decomposes after a single
suspension.
\end{enumerate}

It is the purpose of this note to show that many of these
decompositions carry over into the context of simplicial spaces
which satisfy a mild cofibration condition and to put these in a
coherent picture. Recall that a simplicial space $X_{\bullet}$ is a
set of topological spaces $X_n$, $n \geq 0$, together with
continuous maps $d_i:X_n \to X_{n-1}$ and  $s_j:X_n \to X_{n+1}$
which satisfy the simplicial identities. A natural filtration of
each space $X_n$ is defined next.

\begin{defin}\label{defin:good.simplicial.spaces}
Define subspaces $S^t(X_n) = \cup s_{i_1}s_{i_2}\cdots s_{i_t}(X_{n-t})\subset X_n$
with $S^0(X_n) = X_n$ and $S^1(X_n)=S(X)$ (for notational convenience). 
This defines a natural \textsl{decreasing filtration}
for the spaces $X_n$ in a simplicial space
$X_{\bullet}$, where
$$s_0^n(X_0) = S^{n}(X_{n}) \subset  \cdots \subset S^t(X_{n})\subset \cdots 
\subset S(X_{n})  \subset S^{0}(X_{n}) = X_n$$
and $S^{n+1}(X_n)$ is empty by convention.
\end{defin}

The following concepts appear in \cite{may}, Definition 11.2.

\begin{defin}
A pair of spaces $(X,A)$ is said to be
a \textsl{strong NDR pair} 
provided that there are maps $u:X
\to [0,1]$ and a homotopy $h:X \times [0,1] \to X$ such that
$(X,A)$ is an NDR pair, namely
\begin{enumerate}
  \item   $A=u^{-1}(0)$,
  \item  $h(0,x)=x$ for all $x \in X$,
  \item $h(a,t)\in A$   for all $(a,t) \in A \times[0,1]$,
  \item if $u(x)<1$ then $h(x,1)\in A$.
\end{enumerate} and if $u(x)<1$ then $u(h(x,t))<1$.
\end{defin}

\begin{defin}
A simplicial space $X_{\bullet}$ is said to be \textsl{proper}
if each pair
$(X_{n},S(X_{n}))$ is a strong NDR-pair for all $n$.
\end{defin}

\begin{defin}
A simplicial space $X_{\bullet}$ is said to be \textsl{simplicially
NDR} if each  
$$(S^{t-1}(X_n), S^t(X_n))$$ 
is an NDR pair for
all $t-1\ge 0$ and all $n$.
\end{defin}

Note that 
every degenerate element $x$ in $X_n$ has a unique decomposition as
$$x = s_{j_r}s_{j_{r-1}}\cdots s_{j_1}(y)$$ where $y$ is  in $X_{n-r}$ with $y$ non-degenerate and
$ j_r > j_{r-1}> \cdots > j_1$. Given any sequence $I=
(i_r,i_{r-1}\cdots, i_1)$, write $s_I(X_{n-r}) =
s_{i_r}s_{i_{r-1}}\cdots s_{i_1}(X_{n-r})$ with $|I| = r$. 

\begin{defin}
The
sequence $I= (i_r,i_{r-1}\cdots, i_1)$ is said to be \textsl{admissible}
provided $i_r > i_{r-1} > \cdots > i_1$. In case $I$ is admissible,
define $\widehat{s_I(X_{n-r})} = s_I(X_{n-r})/s_IS(X_{n-r})$.
\end{defin}

The point-set topological properties of $X_{\bullet}$ are basic in
these results. One instance is illustrated by the natural
inclusion $\iota:S(X_{n}) \to X_{n}$ with mapping cone denoted
$K(\iota)$. The proof of the main Theorem \ref{thm:splittings}
implies the suspension of $K(\iota)$ is a retract of the suspension
$\Sigma(X_n)$. On the other-hand, the quotient space $X_n/S(X_{n})$
sometimes has independent useful features such as the case in
\cite{ac} where these spaces are sometimes identified as natural
Spanier-Whitehead duals of certain choices of Lie groups. To ensure
that the properties of $X_n/S(X_{n})$ are reflected in the structure
of Theorem \ref{thm:splittings}, it is useful to know that the
inclusion $S(X_{n}) \to X_{n}$ is a cofibration. 

The precise point-set topology for $|X_{\bullet}|$ admits several
natural choices. Milnor originally topologized $|X_{\bullet}|$ by
the natural quotient topology \cite{milnor2}. Milgram topologized
$BG$, the geometric realization of a simplicial space similarly
\cite{milgram}. Subsequently, Steenrod topologized $BG$ by the
natural compactly generated topology \cite{St}. Finally, May
topologized $|X_{\bullet}|$ with the natural compactly generated,
weak Hausdorff topology which is both elegant and convenient.
This topology is used throughout the current article.

\begin{thm}\label{thm:splittings}
Assume that the simplicial space $X_{\bullet}$ is
simplicially NDR. Then the
spaces $X_n$ in the simplicial space $X_{\bullet}$ are naturally
filtered where
$$s_0^n(X_0) = S^{n}(X_{n}) \subset  \cdots  S^r(X_{n}) \subset \cdots  \subset S(X_{n}) \subset S^{0}(X_{n}) = X_n.$$
Furthermore, these filtrations are split up to homotopy after
suspending once. Thus there are homotopy equivalences which are
natural for morphisms of simplicial spaces
\begin{enumerate}
\item $$\Theta(n): \Sigma(X_n) \longrightarrow  \bigvee_{0 \leq r \leq n} \Sigma(S^r(X_{n})/S^{r+1}(X_{n})),$$

\item $$H(n): \Sigma(X_n) \longrightarrow  \bigvee_{0 \leq r \leq n}\bigvee_{J} \Sigma(\widehat{s_J(X_{n-r})})$$
where $$J= (j_r,j_{r-1}\cdots, j_1)$$ is admissible with $|J| =
r$ and $ 0 \leq r \leq n$ and

\item the map $H(n)$ restricts to a homotopy equivalence
$${H(n)|}_t: \Sigma(S^t(X_n)) \longrightarrow  \bigvee_{t \leq r \leq n}\bigvee_{J} \Sigma(\widehat{s_{J}(X_{n-r})}).$$
\end{enumerate}
\end{thm}

\noindent\textbf{Remarks:}\label{remark:natural.transformation}

\noindent (1) The splitting maps above in Theorem \ref{thm:splittings}
are induced by the natural transformation from the identity to the
decomposition maps $\Theta(n)$ regarded as functors from simplicial
spaces to spaces.

\noindent (2) The finer decompositions obtained using the maps $H(n)$
arise from spaces
$\widehat{s_{J_t}(X_{n-t})}$ with fixed $t = |J_t|$. In case $t$ is
fixed, the spaces $\widehat{s_{J_t}(X_{n-t})}$ are homeomorphic, but
not equal.

\noindent (3) The notation in the proof and statement of 
Theorem \ref{thm:splittings} simplifies
considerably if for fixed $t$, these differences of the
$\widehat{s_{J_t}(X_{n-t})}$ are not addressed. Since the proof that
moment-angle complexes admit stable decompositions in
\cite{bahri.bendersky.cohen.gitler} uses an analogous proof which
keeps track of these differences, the more technically complicated
statement as well as proof are retained here.

The following was proved by J.~P.~May as Lemma 11.3 \cite{may}.
\begin{prop}\label{prop:quotients.in.realization}
Assume that the simplicial space $X_{\bullet}$ is proper. 
Then the geometric realization
$|X_{\bullet}|$ is naturally filtered by $F_j|X_{\bullet}|$ with
induced homeomorphisms
$$\Sigma^{j}(X_j/S(X_j)) \to F_j|X_{\bullet}|/F_{j-1}|X_{\bullet}|.$$
\end{prop}

The next corollary follows from Theorem \ref{thm:splittings} and
Proposition \ref{prop:quotients.in.realization}. Notice that
Corollary \ref{cor:finer.splitting.quotients.realization} implies
that the stable summands in Theorem \ref{thm:splittings} are all
given in terms of the filtration quotients
$F_j|X_{\bullet}|/F_{j-1}|X_{\bullet}|$. In addition, the natural
$d^1$-differential in homology arising from the natural spectral
sequence first investigated by G.~Segal \cite{segal} then admits a
geometric interpretation in terms of this decomposition, a point not
developed here.

\begin{cor}\label{cor:finer.splitting.quotients.realization}
Assume that the simplicial space $X_{\bullet}$ is proper and simplicially
NDR. 
Then there are natural homotopy
equivalences $$K(n,t): \Sigma^{n+1}(S^t(X_n)/S^{t+1}(X_n)) 
\longrightarrow
\bigvee_{J_t}\Sigma^{t+1}(F_{n-t}|X_{\bullet}|/F_{n-t-1}|X_{\bullet}|)$$ 
where $J_t$
runs over all admissible sequences with $t = |J_t|$
and $t$ is a fixed integer such that $0\le t\le n$. Thus by Theorem
\ref{thm:splittings}, there are natural homotopy equivalences
$$\Theta(n): \Sigma^{n+1}(X_n) 
\longrightarrow \bigvee_{0 \leq t \leq n}
\bigvee_{J_t }\Sigma^{t+1}
(F_{n-t}|X_{\bullet}|/F_{n-t-1}|X_{\bullet}|)$$ 
where $J_t$ runs over
all admissible sequences with $t = |J_t|$.

\end{cor}

The authors would like to thank Tom Baird for discussions concerning
real algebraic sets and Phil Hirschhorn for discussions concerning
simplicial spaces.

\section{Simplicial spaces}\label{sec:Simplicial spaces}

The purpose of this section is to recall standard properties of
simplicial spaces to be used in the proof of the main theorem.
Throughout this article $X_{\bullet}$ is assumed to be a simplicial
space which is proper.
Standard properties are stated in the next
lemma.
\begin{lem} \label{lem:wedges}
If $X_{\bullet}$ is a simplicial space which is 
simplicially NDR
then it satisfies the following properties.
\begin{enumerate}

\item The pair $$(s_{i_r}s_{i_{r-1}}\cdots s_{i_1}S^{t-1}(X_{n}), s_{i_r}s_{i_{r-1}}\cdots s_{i_1} S^t(X_n))$$ is an NDR pair
for all sequences $(i_r,i_{r-1}\cdots, i_1)$  and all $n$ and $t$
with $r \leq n-t.$ Thus the maps
$$s_{j_r}s_{j_{r-1}}\cdots s_{j_1} S^t(X_n) \to s_{j_r}s_{j_{r-1}}\cdots s_{j_1} S^{t-1}(X_{n})$$
are cofibrations for all sequences $(i_r,i_{r-1}\cdots, i_1)$
and all $n$ and $t$ with $r \leq n-t$ .

\item If $0 \leq r \leq n-1$, there are homeomorphisms which are
natural for morphisms of simplicial spaces

$$\gamma(n,r):\vee_{J}
\widehat{s_J(X_{n-r})} \longrightarrow S^r(X_n)/S^{r+1}(X_n)$$ where
    \begin{enumerate}
    \item $J= (j_r,j_{r-1}\cdots, j_1)$ is admissible with $|J| = r$,
    \item $\widehat{s_J(X_{n-r})} = s_J(X_{n-r})/s_JS(X_{n-r})$
    and
    \item $S^n(X_n)$ is equal to $s_0^n(X_0)$.
    \end{enumerate}
\end{enumerate}
\end{lem}

\begin{proof}
The pair $(S^{t-1}(X_{n}), S^{t}(X_{n}))$ is an NDR pair for all
$n$ and $t$ with $ t \leq n $ by hypothesis. Let $I=
(i_r,i_{r-1}\cdots, i_1)$. Since the map $s_I =
s_{i_r}s_{i_{r-1}}\cdots, s_{i_1}$ is a homeomorphism onto its
image with one choice of inverse $d_{i_1} \cdots
d_{i_{r-1}}d_{i_r}$, the pair
$$(s_{i_r}s_{i_{r-1}}\cdots s_{i_1}S^{t-1}(X_{n}),
s_{i_r}s_{i_{r-1}}\cdots s_{i_1} S^t(X_n))$$ 
is an NDR pair for all
sequences.
The first part of Lemma \ref{lem:wedges} follows.

To prove the second part of Lemma \ref{lem:wedges}, observe that
$S^r(X_n) = \cup_I s_{i_r}s_{i_{r-1}}\cdots s_{i_1}(X_{n-r})$ for
$I= (i_r,i_{r-1}\cdots, i_1)$ admissible with $|I|= r$. Now 
consider $J = (j_r,j_{r-1}\cdots, j_1)$ admissible for $I \neq J$.
Since $I \neq J$, let $t$ denote the largest integer for which
the entries $i_t$ and $j_t$ are not equal; without loss of
generality we can assume that $i_t > j_t$.
Applying $d_{i_t} d_{i_{t+1}} \cdots d_{i_{r-1}}d_{i_r}$ and using
the simplicial identities gives that
$s_I(X_{n-r}) \cap s_J(X_{n-r}) \subset S^{r+1}(X_n)$.

Now if $J = (j_r,j_{r-1}\cdots, j_1)$ is admissible with $|J| =
r$, then the inclusions give rise to a relative homeomorphism
$$F(n,r):(\sqcup_I s_J(X_{n-r}), \sqcup_J s_J(S(X_{n-r})))
\to (S^r(X_n), S^{r+1}(X_n))$$
which thus induces a
natural map
$$\gamma(n,r):\vee_{J}
\widehat{s_{J}(X_{n-r})} \to S^r(X_n)/S^{r+1}(X_n)$$ 
that is a
continuous bijection.
Hence, it suffices to check that
the map $\gamma(n,r)$ is open.
Note that there is a commutative diagram
\[
\begin{CD}
\sqcup_J s_J(X_{n-r})  @>{F(n,r)}>> S^r(X_n) \\
  @V{\pi_1}VV            @VV{\pi_2}V \\
\vee_{J}\widehat{s_{J}(X_{n-r})} @>{\gamma(n,r)}>> 
S^r(X_n)/S^{r+1}(X_n).
\end{CD}
\]
where $\pi_1$ and $\pi_2$ are the natural projection maps.
$F(n,r)$ is an open map as it is a local homeomorphism,
and so it follows that
the induced map $\gamma(n,r)$ is
also open.
The second part of Lemma
\ref{lem:wedges} follows.
\end{proof}

\section{The Proof of Theorem \ref{thm:splittings}}\label{sec:Proof of Theorem splitting.in.section.simplicial.spaces}

Theorem \ref{thm:splittings} gives two
different decompositions:
\begin{enumerate}
  \item One decomposition arises from the equivalence $$\Theta(n): \Sigma(X_n) \longrightarrow  \bigvee_{0 \leq r \leq n}
\Sigma(S^r(X_{n})/S^{r+1}(X_{n})).$$

  \item The second decomposition arises by using the equivalences 
$$\widehat {{H(n)|}_t}: \Sigma(S^t(X_n)/S^{t+1}(X_n)) \longrightarrow
\Sigma(\bigvee_{J_t}\widehat{s_{J_t}(X_{n-t})})$$
induced by the ${H(n)|}_t$.

\end{enumerate} One direct proof arises from giving both splittings at once as given below.
This proof is the precise setting of the analogue of classical
James-Hopf invariant maps and how they fit into a simplicial setting
as well as a splitting of simplicial spaces.

The details of proof come from a construction of the maps ${\widehat {{H(n)|}_t}}$
together with some tedious verifications using simplicial
identities. The main work requires definitions of the analogue of
James-Hopf invariants.

Let $$D(n,r) = \vee_{|J| = r} \widehat{s_J(X_{n-r})} $$ where
\begin{enumerate}
\item $J = (j_r,j_{r-1}\cdots, j_1)$ is admissible,
\item $\widehat{s_{J}(X_{n-r})} = s_{J}(X_{n-r})/s_{J}S(X_{n-r})$ and
\item $S^n(X_n)$ is equal to $s_0^n(X_0)$.
\end{enumerate}

The method of proof is to exhibit a map
$$H(n): \Sigma(X_n) \to \bigvee_{0 \leq r \leq n} \Sigma(D(n,r))$$
as described next with the following properties.

\begin{lem}\label{lem:filtration.preserving.H}
Assume that the simplicial space $X_{\bullet}$ is 
simplicially
NDR. Then
there is a map
$$H(n): \Sigma(X_n) \to \bigvee_{0 \leq r \leq n} \Sigma(D(n,r))$$
with the following properties.

\begin{enumerate}

\item The map $H(n)$ restricts to a map
$${H(n)|}_t: \Sigma(S^t(X_n)) \to \bigvee_{t \leq r \leq n} \Sigma(D(n,r)).$$

\item There is  a morphisms of cofibrations
\[
\begin{CD}
\Sigma(S^{t+1}(X_n))  @>{{H(n)|}_{t+1}}>> \bigvee_{t+1 \leq r \leq n} \Sigma(D(n,r)) \\
  @V{i(n,t+1)}VV            @VV{\bar i(n,t+1)}V \\
\Sigma(S^t(X_n))  @>{{H(n)|}_t}>> \bigvee_{t \leq r \leq n} \Sigma(D(n,r)) \\
  @V{q(t)}VV            @VV{\bar q(t)}V \\
  \Sigma(S^t(X_n)/S^{t+1}(X_n))  @>{\widehat {{H(n)|}_t}}>> \Sigma(D(n,t))
\end{CD}
\]

\noindent where $i(n,t+1)$, $\bar i(n,t+1)$, $q(t)$ and $\bar q(t)$
are the natural inclusions and projections.

\item The map $\widehat {{H(n)|}_t}:\Sigma(S^t(X_n)/S^{t+1}(X_n)) \to
\Sigma(D(n,t))$ is induced by $H(n)$.

\item The map $\widehat {{H(n)|}_t}$ is a homotopy equivalence, and

\item the map ${H(n)|}_{t+1}:\Sigma(S^{t+1}(X_n)) \to \bigvee_{t+1 \leq r \leq n}\Sigma(D(n,r))$
is a homotopy equivalence by downward induction on $t$ starting with
$t=n$ which is the equivalence $\widehat {{H(n)|}_n}:\Sigma(S^n(X_n)
= s_0^n(X_0) \to \Sigma(D(n,n)).$
\end{enumerate}
\end{lem}

Notice that Theorem \ref{thm:splittings}
is an immediate consequence of Lemma
\ref{lem:filtration.preserving.H}. 
The key step is to define the map $H(n)$; the verification of
its properties will be left to the reader.

Consider admissible
sequences $I= (i_r,i_{r-1}\cdots, i_1)$ with $ n \geq i_r > i_{r-1}>
\cdots> i_1 \geq 0$. Thus $s_I(X_{n-r})$ is a subspace of $X_n$.
Define $\chi(I)= (i_1,i_2, \cdots i_{r-1}, i_r)$. Thus
$$d_{\chi(I)}= d_{i_1}d_{i_2} \cdots d_{i_{r-1}}d_{i_r} 
~~\rm{and}~~
d_{\chi(I)}\circ s_I(x) = x.$$
The natural lexicographical total ordering on such admissible
sequences is obtained next from a partial ordering on all such
sequences (not necessarily admissible).
\begin{defin}\label{defin:order.admissible.sequences}
If $I= (i_r,i_{r-1}\cdots, i_1)$ and $J = (j_r,j_{r-1}\cdots, j_1)$
are sequences with $|I| = |J| =  r$ and $I \neq J$, define
$I < J$ provided there exists a $p \leq r$ such that
$i_k = j_k$ if $p < k \leq r$ with
$i_p < j_p$.
\end{defin}

Since an admissible sequence $I= (i_r,i_{r-1}\cdots, i_1)$ satisfies
$n \geq i_r > i_{r-1}> \cdots> i_1 \geq 0$, there are exactly
${\binom {n+1} r}$ choices of admissible sequences with $|I| = r$.
Furthermore, the partial ordering in Definition
\ref{defin:order.admissible.sequences} restricts to a total ordering
on admissible sequences which satisfy $|I| = r$. The next lemma is a
direct verification with details omitted.

\begin{lem} \label{lem:upper.triangular.d.I.s.J}
Assume that $X_{\bullet}$ is a simplicial space, both $I$ and $J$ are
admissible with $|I| = |J| = r$ and that $x \in X_{n-r}$. Then

\[
d_{\chi(I)}s_J(x) =
\begin{cases}
x & \text{if $I = J$ and}\\
s_{m}(y) & \text{for some $s_m$ and some $y$ if $I < J $.}
\end{cases}
\]
\end{lem}

\begin{defin}\label{defin:product.of.faces.map}
Restrict to admissible sequences $I_s$ with $|I_s| = r$. Define
$$\D(r): X_n \to (X_{n-r})^{\binom {n+1} r}$$ by
$$\D(r) = d_{\chi(I_1)} \times d_{\chi(I_2)} \times \cdots \times d_{\chi(I_{\alpha(n,r)})}$$
where $I_s < I_{s+1} $ for all $1 \leq s$ and $\alpha(n,r)= {\binom
{n+1} r}.$

If $r=0$, then $$\D(0): X_n \to X_{n}$$ is the identity map by
convention.

\end{defin}

The next lemma follows at once from the definitions.
\begin{lem} \label{lem:Filtrations.and.D}
Assume that $X_{\bullet}$ is a simplicial space. Then the map
$$\D(r): X_n \to (X_{n-r})^{\binom {n+1} r}$$ restricts to a map
$$\D(r): S^{r+1}(X_n) \to (S(X_{n-r}))^{\binom {n+1} r}$$
for all $0 \leq r \leq n$. Thus there is a commutative diagram
\[
\begin{CD}
S^{r+1}(X_{n})  @>{ \D(r)}>> (S(X_{n-r}))^{\binom {n+1} r} \\
  @V{i(n,r+1)}VV            @VV{i(n-r,1)^{\binom {n+1} r}}V \\
S^r(X_{n})  @>{\D(r)}>> (X_{n-r})^{\binom {n+1} r} \\
  @V{i(n,r)}VV            @VV{1}V \\
X_n  @>{\D(r)}>> (X_{n-r})^{\binom {n+1} r}.
\end{CD}
\]

\end{lem}

The definition of the map $H(n)$ is given next. Recall the maps
$$\D(r): X_n \to (X_{n-r})^{\binom {n+1} r}$$ given by $$\D(r) =
d_{\chi(I_1)} \times d_{\chi(I_2)} \times \cdots \times
d_{\chi(I_{\alpha(n,r)})} $$ where $I_s < I_{s+1}$ for all $s \geq
1$. The coordinates in $(X_{n-r})^{\binom {n+1} r}$ are indexed by
$\chi(I)$ for $I$ admissible with $|I|=r$. Let $$P_{\chi(I)}:
(X_{n-r})^{\binom {n+1} r} \to X_{n-r}$$ denote the projection map
to the $\chi(I)$-th coordinate.

Recall that $$D(n,r)=\vee_{|J|=r}\widehat{s_{J}(X_{n-r})}$$
where
\begin{enumerate}
\item $J = (j_r,j_{r-1}\cdots, j_1)$ is admissible,
\item $\widehat{s_{J}(X_{n-r})} = s_{J}(X_{n-r})/s_JS(X_{n-r})$ and
\item $S^n(X_n)$ is equal to $s_0^n(X_0)$.
\end{enumerate}

\begin{defin}\label{defin:H(n)}
\begin{enumerate}
  \item Let $$\nu(n,J): \Sigma((X_{n-r})^{\binom {n+1} r}) \longrightarrow
\widehat{s_{J}(X_{n-r})}$$ denote the composite
$\nu(n,J) = \bar \sigma_{J} \circ P_{\chi(J)}$.
  \item Let $$\lambda(n,J): \Sigma(X_n)\longrightarrow
\widehat{s_{J}(X_{n-r})}$$ denote the composite
$\lambda(n,J) = \nu(n,J) \circ \Sigma(\D(r))$.
  \item Let $$ \Phi(n,r):\Sigma(X_n) \longrightarrow \bigvee_{|J|=r}
\widehat{s_{J}(X_{n-r})}$$ denote the sum
$$ \Phi(n,r)= \sum_{|J|= r}\lambda(n,J)$$ where the index
is over all admissible sequence $J$ with a fixed order of
summation.
\item
Define $$H(n): \Sigma(X_n) \to \bigvee_{0 \leq r \leq n}
\Sigma(D(n,r))$$ as the sum $$H(n)= \sum_{0 \leq r \leq n}
\Phi(n,r)$$ with a fixed order of summation.
\end{enumerate}
\end{defin}

Therefore this defines the desired map
$$H(n): \Sigma(X_n) \to \bigvee_{0 \leq r \leq n}
\Sigma(D(n,r))$$ which when restricted to $S^t(X_{n})$ 
makes the following diagram commute (up to homotopy)

\[
\begin{CD}
\Sigma(S^{t}(X_n))  @>{{H(n)|}_{t}}>> \bigvee_{t\leq r \leq n} \Sigma(D(n,r)) \\
  @V{i(n,t)}VV            @VV{\bar i(n,t)}V \\
\Sigma(X_n)  @>{H(n)}>> \bigvee_{0 \leq r \leq n} \Sigma(D(n,r)).
\end{CD}
\]

The desired properties of $H(n)$ can be readily verified and are left
to the reader.

\section{Simplicial sets and real algebraic sets}\label{sec:Simplicial sets and real algebraic sets}

The purpose of this section is to (i) recall standard properties of
simplicial complexes as well as (ii) the way in which Theorem
\ref{thm:splittings} can be applied to real algebraic varieties.

Let $K$ denote an abstract simplicial complex on $m$ vertices
labeled by the set $$[m]=\{1,2,\ldots, m\}.$$ Thus a simplex
$\sigma$ of $K$ is given by an ordered sequence $$\sigma
=(i_1,\cdots, i_k)$$ with $1 \leq i_1 <\cdots < i_k \leq m$ such
that if $\tau \subset \sigma$, then $\tau$ is a simplex of $K$.
Recall the simplicial set $\Delta(K)$ obtained from an abstract
simplicial complex as defined in \cite{bousfield.kan} page 234.

\begin{defin}\label{defin:moment.angle.complex}
The simplicial set $\Delta(K)$ is defined as follows.
\begin{itemize}
\item $\Delta (K)$ 
has $n$-simplices given by the $(n+1)$-tuples of
vertices $(v_0,v_1, \cdots, v_n)$ for which $v_0
\leq v_1 \leq  \cdots \leq  v_n$ 
\item 
the face and degeneracy operators are given by
        $$d_i(v_0,v_1, \cdots, v_n) = (v_0,\cdots v_{i-1}, v_{i+1},\cdots,
        v_n)$$
        and
        $$s_i(v_0,v_1, \cdots, v_n) = (v_0,\cdots v_{i}, v_{i},\cdots,
        v_n).$$
        \end{itemize}
\end{defin}

As pointed out in \cite{bousfield.kan}, the following result follows 
from the two paragraphs before Theorem 15 on page 111 in \cite{spanier}.

\begin{thm}\label{thm:homeomorphic.realization.of.order.complex}
The geometric realization  $|\Delta(K)|$ is homeomorphic to $|K|$.
\end{thm}

The benefits of this construction are expressed in the following
result, which is left to the reader for verification.

\begin{prop}\label{prop:order.complex}
The simplicial space $\Delta(K)$ is proper and simplicially NDR.
Thus Theorem
\ref{thm:splittings} applies to $\Delta(K)$.
\end{prop}

Definitions and basic properties of real algebraic and
semi-algebraic varieties are listed next with main reference
\cite{bcr}.

\begin{defin}\label{defin:real.algebraic.set}
An affine real algebraic set is the the common zero set of a finite
number of real polynomials in $\R[x_1,...,x_k]$ for some $k\ge 1$.
A real algebraic set has a topology induced from $\R^k$ equipped
with the Euclidean topology which is called the
classical topology.\footnote{The Zariski
topology, for which closed sets are given by algebraic subsets, will
not be used here.}
A real semi-algebraic set is a subset of $\R^k$ for some $k$, which
is a finite union of sets each determined by a finite number of
polynomial inequalities.
\end{defin}

The following theorem was proven in \cite{bcr} section 9.4.1.
\begin{thm}\label{tri}
If $X$ be a compact semi-algebraic set then it is triangulable,
i.e. it is homeomorphic
to the geometric realization of a finite simplicial complex.
\end{thm}

Combining the fact that a compact semi-algebraic set is triangulable
with Theorem \ref{thm:splittings} above, the next corollary
follows at once.

\begin{cor}\label{cor:real.varieties.and.order.complex}
Let $X$ be a compact semi-algebraic set which is the geometric
realization of a finite simplicial complex $K$ with order complex
$\Delta(K)$. Then $\Delta(K)$ is a simplicial space for which the
$n$-space $\Delta(K)_n$ admits a decomposition after suspension
given in Theorem \ref{thm:splittings}.
\end{cor}

\section{Examples and Problems}\label{sec:Classical examples}

This section gives a list of examples of decompositions which arise
from the above method.

\begin{exm} \label{exm:spaces.of.homomorphisms}

Let $G$ denote a closed subgroup of $GL_r(\C)$
and consider the 
space of ordered commuting $n$--tuples in $G$, denoted 
$Hom(\oplus_n~\Z, G)$ and topologized as 
a subspace of the product $G^n$. 
These spaces can be assembled to form a simplicial space
which is proper and simplicially NDR, 
and thus 
Theorem \ref{thm:splittings} can be applied,
yielding in particular the decompositions described in \S 1
(see \cite{ac} and \cite{acg} for details).

\end{exm}

\begin{exm} \label{exm:spaces.of.representations}

The group $G$ acts on $Hom(\oplus_n~\Z, G)$ by conjugation with
orbit space denoted $Rep(\oplus_n~\Z, G)$. These spaces also form a
simplicial space satisfying the hypotheses of \ref{thm:splittings}
and thus admit analogous stable decompositions in case $G$ is
a compact Lie group (see \cite{acgo}).

\end{exm}

\begin{exm} \label{exm:moment.angle.complexes}

A product of spaces $X_1 \times X_2 \times \cdots X_n$ decomposes as
a wedge of suspensions after suspending once. These decompositions
induce related decompositions of moment-angle complexes, in special
cases, homotopy equivalent to various varieties obtained from
complements of coordinate planes in Euclidean space
(see \cite{bahri.bendersky.cohen.gitler}). 
The decompositions of simplicial spaces given here are more general than
the decompositions of generalized moment-angle complexes given in
\cite{bahri.bendersky.cohen.gitler}; however they are also 
coarser. 
\end{exm}

It seems useful to conclude by formulating some problems associated
to the topics in this paper:

\noindent\textbf{Problems:}
\begin{enumerate}

\item Give interesting examples 
for which Corollary \ref{cor:real.varieties.and.order.complex}
can be used to provide 
useful information concerning semi-algebraic sets. 

\item Identify useful conditions for which the splittings of
  Theorem \ref{thm:splittings} imply decompositions for the
  suspension of $|X_{\bullet}|$ for more general simplicial 
spaces $X_{\bullet}$.

\item  Identify useful conditions when $Tot_{\bullet}(X)$ can be stably split.

\end{enumerate}

\end{document}